\documentclass[12pt]{article}
\usepackage{amsmath,amsthm,amssymb,amsfonts}
\usepackage[active]{srcltx}
\usepackage{latexsym}
\usepackage{enumerate}
\usepackage{verbatim}
\usepackage{multicol}
\usepackage{graphicx}
\usepackage{epsfig}
\usepackage{color}

\usepackage{fullpage}

\newtheorem{thm}{Theorem}
\newtheorem{lem}[thm]{Lemma}
\newtheorem{prop}[thm]{Proposition}
\newtheorem{cor}[thm]{Corollary}
\newtheorem{definition}{Definition}

\begin{document}
\small{
\title{Greedy balanced pairs in $N$-free ordered sets}
\author{Imed Zaguia\footnote{Supported by NSERC DDG-2018-00011}\\
{\small Dept of Mathematics and Computer Science }\\
{\small Royal Military College of Canada}\\
{\small P.O.Box 17000, Station Forces}\\
{\small K7K 7B4 Kingston, Ontario CANADA}\\
{\small \textrm{E-mail: imed.zaguia@rmc.ca}}
}
\date{\today}
\maketitle

\begin{abstract} An $\alpha$-greedy balanced pair in an ordered set $P=(V,\leq)$ is a pair $(x,y)$ of elements of $V$ such that the
proportion of greedy linear extensions of $P$ that put $x$ before $y$ among all greedy linear extensions is in the real interval
$[\alpha, 1-\alpha]$. We prove that every $N$-free ordered set which is not totally ordered has a $\frac{1}{2}$-greedy balanced pair.

\end{abstract}

\input{epsf}

\section{Introduction}

Throughout, $P =(V, \leq)$ denotes a \emph{finite ordered set}, that is, a finite set $V$ and a binary relation $\leq$ on $V$, that is, reflexive, antisymmetric and transitive. The \emph{dual} of
$P$ is the ordered set $P^{d}=(V,\leq^{d})$ such that $x\leq^{d} y$ if and only if $y\leq x$. For $x,y\in V$ we say that $x$ and $y$ are \emph{comparable} if $x\leq y$ or $y\leq x$, denoted $x\sim y$; otherwise we say that $x$ and $y$ are \emph{incomparable} and we set $x\nsim y$. A set of pairwise incomparable elements is called an \emph{antichain}. A \emph{chain} is a set of pairwise comparable elements, and is also called a \emph{totally ordered set}. The \emph{width} of $P$ is the maximal cardinality of an antichian.

For $x,y\in V$ we say that $y$ is an \emph{upper cover} of $x$ or that $x$ is \emph{a lower cover} of $y$, denoted $x \prec y$, if $x<y$ (that is $x\leq y$ and $x\neq y$) and there is no element $z\in V$ such that $x<z<y$.

In the sequel, when this does not cause confusion, we shall not distinguish between the order and the ordered set, using the same letter, say $P$, to denote both. A \emph{linear extension} $L$ of $P$ is a total order which contains $P$ i.e. $x\leq y$ in $L$ whenever $x\leq y$ in $P$. We denote by $\mathcal{L}(P)$ the set of all linear extensions of $P$.

For $0<\alpha\leq \frac{1}{2}$, an $\alpha$-\emph{balanced pair} in $P=(V,\leq)$ is a pair $(x,y)$ of elements of $V$ such that the ratio of linear extensions of $P$ that put $x$ before $y$ among all linear extensions, denoted $\mathbb{P}_{P}(x<y)$, is in the real interval $[\alpha, 1-\alpha]$. The $\frac{1}{3}$-$\frac{2}{3}$ Conjecture states that  every finite ordered set which is not a totally ordered has a $\frac{1}{3}$-balanced pair.

If true, the ordered set disjoint sum of a one element chain and a two element chain would show that the result is best possible. The $\frac{1}{3}$-$\frac{2}{3}$ Conjecture first
appeared in a paper of Kislitsyn \cite{ki}. It was also formulated independently by Fredman in about 1975 and again by Linial \cite{li}.

The $\frac{1}{3}$-$\frac{2}{3}$ Conjecture is known to be true for ordered sets of width two \cite{li}, for semiorders \cite{gb1}, for bipartite ordered sets
\cite{tgf}, for 5-thin ordered sets \cite{gb2}, for 6-thin ordered sets \cite{pec}, for $N$-free ordered sets \cite{zaguia1} and for ordered sets whose covering graph is a forest \cite{zaguia2}. See \cite{gb} for a survey.\\

In this paper, we consider a variation of the $\frac{1}{3}$-$\frac{2}{3}$ Conjecture.\\

Let $P=(V,\leq)$ be an ordered set. We recall that an element $v\in V$ is \emph{minimal} if there is no element $u\in V$ such that $u<v$. A \emph{maximal} element in $P$ is a minimal element of $P^d$.

\begin{definition}\label{def:greedy}
A linear extension $L =x_1<\cdots < x_n$ of $P$ is \emph{greedy} if it is constructed inductively as follows:
\begin{enumerate}[(i)]
\item $x_{1}$ is a minimal element of $P$.
\item Suppose that the first $i$ elements $x_1,\cdots , x_i$ have already been chosen and consider the set $U(x_i):=\{v\in V: x_i<v\}$. If
$U(x_i)=\varnothing$ or if $U(x_i)\neq \varnothing$ and no element of $U(x_i)$ is minimal in $P \setminus \{x_1,\cdots , x_i\}$, then
choose $x_{i+1}$ to be any minimal element of $P \setminus \{x_1,\cdots , x_i\}$. Otherwise choose $x_{i+1}$ to be a minimal element of $U(x_{i})$.
\end{enumerate}
\end{definition}

Intuitively, a greedy linear extension is built by always "climbing as high as you can" along the chains.

%\begin{figure}[ht]
%\begin{center}
%\leavevmode \epsfxsize=1.6in \epsfbox{fence.eps}
%\end{center}
%\caption{An ordered set and its greedy linear extensions.}
%\label{fence}
%\end{figure}

\begin{figure}[ht]
\begin{center}
\includegraphics[width=120pt]{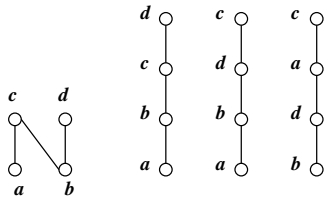}
\end{center}
\caption{An ordered set and its greedy linear extensions.}
\label{fence}
\end{figure}

Deciding whether an ordered set has at least $k$ greedy linear extensions is NP-hard \cite{kirstead-trotter}. Denote by $\mathcal{G}(P)$ the set of all greedy extensions of $P$ (see Figure \ref{fence} for an example). Note that an ordered set and its dual do not necessarily have the same number of greedy linear extensions. Indeed, the ordered set $M$ depicted in Figure \ref{5fence} has eight greedy linear extensions but its dual has four. It should also be noted that the dual of a greedy linear extension is not necessarily a greedy linear extension of the dual ordered set. Indeed, the dual of the greedy linear extension $L = a < b < c < d < e$ of the ordered set $X$ depicted in Figure \ref{5fence} is not a greedy linear extension of $X^d$ (this example is from \cite{rivalzaguia}). An ordered set $P$ is \emph{reversible} if $L^d\in \mathcal{G}(P^d)$ for every $L \in \mathcal{G}(P)$ \cite{rivalzaguia}.

\begin{definition}Let $P=(V,\leq)$ be an ordered set. A $4$-tuple $(a, b, c, d)$ of distinct elements of $V$ is an $N$ in $P$ if $a \prec b\succ c \prec d$ and $a\nsim d$. The ordered set $P$ is $N$-{\it free} if it does not contain an $N$.
\end{definition}
The  class of $N$-free ordered sets was first introduced by Grillet \cite{grillet}. It is a generalization of series-parallel ordered sets (see for example \cite{habib-jegou}).

Rival and Zaguia \cite{rivalzaguia} proved that $N$-free ordered sets are reversible. Since the dual of an $N$-free ordered set is also $N$-free it follows that an $N$-free ordered set and its dual have the same set of greedy linear extensions.

%\begin{figure}[ht]
%\begin{center}
%\leavevmode \epsfxsize=1.8in \epsfbox{5fence.eps}
%\end{center}
%\caption{Non reversible ordered sets.}
%\label{5fence}
%\end{figure}

\begin{figure}[ht]
\begin{center}
\includegraphics[width=120pt]{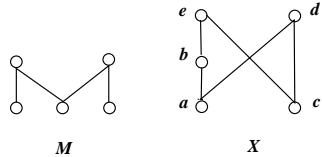}
\end{center}
\caption{Non reversible ordered sets.}
\label{5fence}
\end{figure}

We recall that if $y\nleqslant x$, then there exists a greedy linear extension that puts $x$ before $y$ \cite{bhj}. For $0<\alpha\leq \frac{1}{2}$, an $\alpha$-\emph{greedy balanced pair} in $P=(V,\leq)$ is a pair $(x,y)$ of elements of $V$ such that the ratio of greedy linear extensions of $P$ that put $x$ before $y$ among all greedy linear extensions, denoted $\mathbb{GP}_{P}(x<y)$, is in the real interval $[\alpha, 1-\alpha]$. It is then natural to consider the following problem.\\

\noindent\textbf{Problem:} What is the largest value of $\alpha$ such that every finite ordered set which is not totally ordered has an $\alpha$-greedy balanced pair?\\

The example shown in Figure \ref{fence} shows that $\alpha \geq \frac{1}{3}$.\\

Let $I$ be an ordered set such that $|I|\geq 2$ and let $\{P_{i}:=(V_i,\leq_i)\}_{i\in I}$ be a family of pairwise disjoint nonempty ordered sets that are all disjoint from $I$. The \emph{lexicographical sum} $\displaystyle \sum_{i\in I} P_{i}$ is the ordered set defined on
$\displaystyle \bigcup_{i\in I} V_{i}$ by $x\leq y$ if and only if
\begin{enumerate}[(a)]
\item There exists $i\in I$ such that $x,y\in V_{i}$ and $x\leq_i y$ in $P_{i}$; or
\item There are distinct elements $i,j\in I$ such that $i<j$ in $I$,   $x\in V_{i}$ and $y\in V_{j}$.
\end{enumerate}

The ordered sets $P_{i}$ are called the \emph{components} of the lexicographical sum and the ordered set $I$ is the \emph{index set}.
If $I$ is a totally ordered set, then $\displaystyle \sum_{i\in I} P_{i}$ is called a \emph{linear sum}. On the other hand, if $I$ is
an antichain, then $\displaystyle \sum_{i\in I} P_{i}$ is called a \emph{disjoint sum}. Henceforth we will use the symbol $\oplus$ to indicate the linear sum and the symbol $+$ to indicate the disjoint sum.

Let $L =x_1<\cdots < x_n$ be a linear extension of an ordered set $P$. A consecutive pair $(x_i,x_{i+1})$ of elements is a \emph{jump} in $L$ if $x_i$ is not comparable to $x_{i+1}$ in $P$. The jumps induce a decomposition
of the linear extension L into chains of $P$ that we call the \emph{blocks} of $L$. Thus, $L=C_1\oplus C_2 \oplus \cdots$ where
$(\sup C_i, \inf C_{i+1})$, $i = 1,2\cdots$, are the jumps of $L$. Let $s(P, L)$ be the number of jumps in $L$. The minimum number of jumps taken over all linear extensions of an ordered set is called the \emph{jump number}. This parameter has received a lot of attention, see \cite{chein-habib}, \cite{cogis-habib}, \cite{gierz-poguntke},
\cite{pulleyblank},\cite{rival84,rival83}, \cite{rivalzaguia}, \cite{syslo}. The relation between the jump number and greedy linear extensions can be found for example in \cite{zahar} and \cite{rivalzaguia}.

Let $(a_1,a_2,...,a_m)$ be a sequence of nonnegative integers summing to $n$ and suppose we have $m$ categories $K_1,...,K_m$. Following Stanley \cite{stanley} page 16 we denote by $\binom{n}{a_1,a_2,...,a_m}$ the number of ways of assigning each element of an $n$-set $S$ to one of the categories $K_1,...,K_m$ so that exactly $a_i$ elements are assigned to $K_i$.  These are generally referred as \emph{multinomial coefﬁcients}. For the case $m=2$, $\binom{n}{k,n-k}$ is just the binomial coefficient $\binom{n}{k}$. In general we have

\[\displaystyle \binom{n}{a_1,a_2,...,a_m}= \frac{n!}{a_1!\,a_2!\cdots a_m!}.\]

\begin{thm}Let $P=P_1+P_2+\cdots + P_m$. Then
\[|\mathcal{G}(P)|=\displaystyle \sum_{L\in \mathcal{G}(P_1)} \sum_{M\in \mathcal{G}(P_2)}\cdots \sum_{Z\in \mathcal{G}(P_m)} \binom{s(L,P_1)+\cdots+s(Z,P_m)+m}{s(L,P_1)+1,...,s(Z,P_m)+1}\]
\end{thm}
\begin{proof}Let $L=C_1\oplus C_2 \oplus \cdots$ be a greedy linear extension of $P$ where
$(\sup C_i, \inf C_{i+1})$, $i = 1,2\cdots$, are the jumps of $L$. Each $C_i$, $i = 1,2\cdots$, is a chain of $P$ and hence each $C_i$, $i = 1,2\cdots$, is a chain of $P_j$ for some $1\leq j\leq m$. For a fixed $j$ such that $1\leq j\leq m$ we let $C^j_1\oplus C^j_2 \oplus \cdots$ be the blocks of $L$ that are chains in $P_j$ so that $C^j_l$ appears before $C^j_{l'}$ in $L$ if $l<l'$, for all $l,l'$. We claim that $L^j:=C^j_1\oplus C^j_2 \oplus \cdots$ is a greedy linear extension of $P_j$ where
$(\sup C^j_i, \inf C^j_{i+1})$, $i = 1,2\cdots$, are the jumps of $L^j$. Clearly, $L^j$ is a linear extension of $P_j$. If $C^j_i$ and $C^j_{i+1}$ are consecutive in $L$, then $(\sup C^j_i, \inf C^j_{i+1})$ is a jump in $L$ and hence is a jump in $L^j$. Else, $C^j_i$ and $C^j_{i+1}$ are not consecutive in $L$. It follows from our assumption that $L$ is greedy that $\sup C^j_i$ and $\inf C^j_{i+1}$ are incomparable in $P$ and hence $(\sup C^j_i, \inf C^j_{i+1})$ is a jump in $L^j$.

Let $i\geq 1$ be such that $C_i$ and $C_{i+1}$ are chains of two distinct components of $P$ and consider $L^{(i)}$ to be the linear order obtained from $L$ by swapping $C_i$ and $C_{i+1}$. Then $L^{(i)}$ is a greedy linear extension of $P$ if and only if $L$ is a greedy linear extension of $P$. %This proves that $L^j$ must be a greedy linear extension of $P_j$ since $L$ is.
\end{proof}

\begin{cor}Let $P$ be a disjoint sum of $m$ chains. Then $|\mathcal{G}(P)|=m!$.
\end{cor}

%\begin{figure}[ht]
%\begin{center}
%\leavevmode \epsfxsize=1in \epsfbox{disjointsum.eps}
%\end{center}
%\caption{A disjoint sum of ordered sets with no greedy balanced pair in any connected component.}
%\label{disjointsum}
%\end{figure}

\begin{figure}[ht]
\begin{center}
\includegraphics[width=75pt]{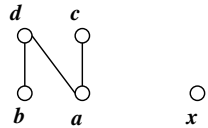}
\end{center}
\caption{A disjoint sum of ordered sets with no $\frac{1}{3}$-greedy balanced pair in any connected component.}
\label{disjointsum}
\end{figure}

It should be noted that for a disconnected ordered set a $\frac{1}{3}$-greedy balanced pair in a component is not necessarily a $\frac{1}{3}$-greedy balanced pair in the ordered set. The example depicted in Figure \ref{disjointsum} has eleven greedy linear extensions. No $\frac{1}{3}$-greedy balanced pair of the $N$ is $\frac{1}{3}$-greedy in the disjoint sum. Indeed, $\mathbb{GP}_{P}(b<a)=\mathbb{GP}_{P}(c<d)=\mathbb{GP}_{P}(b<c)=\frac{8}{11}>\frac{2}{3}$.

The proof of the following proposition is easy and is left to the reader.

\begin{prop}Let $P=P_1\oplus P_2\oplus \cdots \oplus P_k$. Then
\[|\mathcal{G}(P)|=\displaystyle \prod_{i=1}^{k}{|\mathcal{G}(P_i)|}.\]
\end{prop}

Let $P=(V,\leq)$ be an ordered set (not necessarily $N$-free). A subset $A$ of $V$ is called {\it autonomous} (or interval or module) in $P$ if for all $v\not\in A$ and for all $a,a^{\prime} \in A$

\begin{equation}
(v<a\Rightarrow v < a^{\prime})\;\mathrm{and}\;(a<v\Rightarrow a^{\prime} < v).
\end{equation}

For instance, the empty set, the singletons in $V$ and the whole set $V$ are autonomous in $P$.

Let $P$ be an ordered set. A $3$-tuple $(x,y,z)$ of $P$ is \emph{good} if $x<z$, $z\nsim y\nsim x$ and $\{x,y\}$ is autonomous for $P\setminus \{z\}$ (note that in this case we must have $x\prec z$). In \cite{zaguia2}, the author of this note showed that if $(x,y,z)$ is good, then $(x,y)$ is a $\frac{1}{3}$-balanced pair. The example depicted in Figure \ref{disjointsum} shows that this is not true for greedy linear extensions. Indeed, $(a,b,c)$ is a good set yet $(a,b)$ is not a $\frac{1}{3}$-greedy balanced pair.

Our main result is this.

\begin{thm}\label{nfree}Every finite $N$-free ordered set which is not totally ordered has a $\frac{1}{2}$-greedy balanced pair.
\end{thm}

In particular, an $N$-free ordered set which is not a chain has necessarily an even number of greedy linear extensions.

We mention that every finite ordered set can be embedded into a \emph{finite} $N$-free ordered set (see for example \cite{pz}).
It was proved in \cite{blk} that the number of (unlabeled) $N$-free ordered sets is \[\displaystyle 2^{\, n\log_{2}(n) + o (n\log_{2}(n))}.\]

We propose the following conjecture.\\
\textbf{Conjecture:} Every ordered set of width two has a $\frac{1}{3}$-greedy balanced pair.

\section{$N$-free ordered sets}

The following lemma will be useful. Its proof is easy and is left to the reader.

\begin{lem}\label{lem:cover}Let $P$ be an $N$-free ordered set and $x,y\in P$. If $x$ and $y$ have a common upper cover, then $x$ and $y$ have the same set of upper covers. Dually, If $x$ and $y$ have a common lower cover, then $x$ and $y$ have the same set of lower covers.
\end{lem}

\begin{lem}\label{nfree4} Let $P$ be an $N$-free ordered set which is not totally ordered and suppose that every minimal element $x$ is either maximal in $P$ or $x$ has an upper cover which is not minimal in $P\setminus \{x\}$. Then $P$ has two distinct minimal elements $m,m'$ such that $\{m,m'\}$ is autonomous in $P$.
\end{lem}
\begin{proof}If every minimal element of $P$ is also maximal, then $P$ is an antichain and the required conclusion follows and we are done. Otherwise, there exists a minimal element $x$ of $P$ which has an upper cover $y$ which is not minimal in $P\setminus \{x\}$. There exists then a minimal element $z$ distinct from $x$ such that $z<y$. If $z$ is a lower cover of $y$, then it follows from Lemma \ref{lem:cover} that $x$ and $z$ have the same set of upper covers. Clearly, $\{x,z\}\subseteq Min(P)$ and $\{x,z\}$ is autonomous in $P$. If $z$ is not a lower cover of $y$, then let $t$ be an upper cover of $z$ such that $z<t<y$. We claim that $t$ is not minimal in $P\setminus \{z\}$. If not, since $z$ is minimal and not maximal, $z$ has an upper cover $z'\neq t$ which is not minimal in $P\setminus \{z\}$. Hence, $z'$ has a lower cover $z''$ distinct from $z$. It follows from Lemma \ref{lem:cover} that   $z''$ is also a lower cover of $t$ contradicting our assumption that $z$ is the unique lower cover of $t$. This proves our claim. The previous argument applies to any minimal element $z$ distinct from $x$ such that $z<y$. Among all these  $z$ choose one laying on a chain of maximum cardinality containing $y$ and let $t'$ be a lower cover of $t$ distinct from $z$, where $t$ is an upper cover of $z$ such that $z<t<y$. Note that $t'\neq x$ because $x$ is a lower cover of $y$. Then $t'$ is minimal in $P$ because otherwise a minimal element $t''<t'$ verifies $t''$ distinct from $x$ and $t''<y$ and $t''$ lays on a chain containing $y$ of cardinality larger than that containing $z$, contradicting our choice of $z$. Finally we have $\{t',z\}\subseteq Min(P)$ and $\{t',z\}$ is autonomous in $P$. This completes the proof of the lemma.
\end{proof}

\section{A Proof of Theorem \ref{nfree}}

Let $P=(V,\leq)$ be an ordered set. A map $f$ from $V$ to $V$ is \emph{order preserving} if for all $x,y\in V$, $x\leq y$ implies $f(x)\leq f(y)$. An \emph{automorphism} of $P$ is a bijective map $f$ such that $f$ and $f^{-1}$ are order preserving.

\begin{lem}\label{lem:automo}Let $P=(V,\leq)$ be an ordered set and $f$ be an automorphism of $P$. The following propositions are true.
\begin{enumerate}[$(1)$]
  \item $f$ maps minimal elements of $P$ onto minimal elements of $P$.
  \item $f$ maps maximal elements of $P$ onto maximal elements of $P$.
  \item For all $x\in V$, $f(U(x))=U(f(x))$.
\end{enumerate}
\end{lem}
\begin{proof}Propositions (1) and (2) are straightforward. The proof is left to the reader.\\
(3) Let $t\in f(U(x))$. There exists $y>x$ such that $t=f(y)$. Since $f$ is order preserving and one-to-one we infer that $t=f(y)>f(x)$ proving $t\in U(f(x))$ and hence $f(U(x))\subseteq U(f(x))$. Similarly, and since $f^{-1}$ is one-to-one and order preserving we obtain $U(f(x)) \subseteq f(U(x))$.
\end{proof}

\begin{lem}\label{lem:image}Let $P=(V,\leq)$ be an ordered set, $f$ be an automorphism of $P$ and $L=x_1<\cdots <x_n$ be a greedy linear extension of $P$. Then $f(L):=f(x_1)<\cdots <f(x_n)$ is also a greedy linear extension of $P$.
\end{lem}
\begin{proof}Clearly, $f(L)$ is a total order. Let $x<y$ in $P$. Since $f^{-1}$ is order preserving and one-to-one we infer that $f^{-1}(x)<f^{-1}(y)$ in $P$. Because $L$ is a linear extension of $P$ it follows that $f^{-1}(x)<f^{-1}(x)$ in $L$. Hence, $x=f(f^{-1}(x))<f(f^{-1}(y))=y$ in $f(L)$ proving that $f(L)$ is a linear extension of $P$.

We now prove that $f(L)$ is greedy by showing that $f(L)$ is obtained following the algorithm described in Definition \ref{def:greedy}. It follows from (1) of Lemma \ref{lem:automo} and $x_1$ is a minimal element of $P$ that $f(x_1)$ is a minimal element of $f(L)$. Suppose that the first $i$ elements $f(x_1),\cdots , f(x_i)$ have been chosen following the algorithm described in Definition \ref{def:greedy} and consider the set $U(f(x_i))$,  which by (3) of Lemma \ref{lem:automo}, is equal to $f(U(x_i))$. \underline{Suppose} $U(x_i)=\varnothing$ or $U(x_i)\neq \varnothing$ and no element of $U(x_i)$ is minimal in $P \setminus \{x_1,\cdots , x_i\}$. Then $U(f(x_i))=\varnothing$ or $U(f(x_i))\neq \varnothing$ and no element of $U(f(x_i))$ is minimal in $P \setminus \{f(x_1),\cdots , f(x_i)\}$. Since $L$ is greedy $x_{i+1}$ is a minimal element of $P \setminus \{x_1,\cdots , x_i\}$. Clearly, $f(x_{i+1})$ is a minimal element of $P \setminus \{f(x_1),\cdots , f(x_i)\}$. \underline{Otherwise}, $x_{i+1}$ is a minimal element of $U(x_{i})$ and hence $f(x_{i+1})$ is a minimal element of $U(f(x_{i}))$. This proves that $f(L)$ is obtained following the algorithm described in Definition \ref{def:greedy} and therefore is greedy. This completes the proof of the lemma.
\end{proof}

\begin{cor}\label{l2} Let $P$ be an ordered set that has an autonomous set $A$ such that $|A|\geq 2$ and $A$ is an antichain. Then any pair of distinct elements of $A$ is a $\frac{1}{2}$-greedy balanced pair.
\end{cor}
\begin{proof}Assume $|A|\geq 2$ and let $x,y\in A$. Let $L$ be a greedy linear extension that puts $x$ before $y$. Consider the map $f$ defined as follows: $f(x)=y$, $f(y)=x$ and $f(t)=t$ if $t\not \in \{x,y\}$. Clearly $f$ is a bijection. It follows from our assumption that $A$ is autonomous that $f$ and $f^{-1}$ are order preserving and hence $f$ is an automorphism of $P$. Applying Lemma \ref{lem:image} we obtain that $f(L)$ is a greedy linear extension of $P$. Note that $f(L)$ is obtained from $L$ by swapping the positions of $x$ and $y$. The swapping operation defines a bijection from the set of greedy linear extensions that put $x$ before $y$ onto the set of greedy linear extensions that put $y$ before $x$. Therefore, ${\mathbb{GP}}_{P}(x<y)=\frac{1}{2}$.
\end{proof}

\begin{lem}\label{nfree2} Let $P$ be an ordered set (not necessarily $N$-free) and suppose that $P$ has a
minimal element $a$ such that $a$ is not maximal in $P$ and every upper cover of $a$ is minimal in $P\setminus \{a\}$. Let $x$ and $y$
be two incomparable elements of $P\setminus \{a\}$. Then the number of greedy linear extensions of $P$ that put $x$ before $y$ equals the number of greedy linear extensions of $P\setminus \{a\}$ that put $x$ before $y$.
\end{lem}
\begin{proof}Let $L$ be a greedy linear extension of $P$. We claim that $L\setminus \{a\}$ is a greedy linear extension of $P\setminus \{a\}$. Clearly, $L\setminus \{a\}$ is a linear extension of $P\setminus\{a\}$. Also, the successor $a^{*}$ of $a$ in $L$ (a cannot be maximal in $L$ since it is not maximal in $P$) must be an upper cover of $a$ in $P$ since every upper cover of $a$ is minimal in $P\setminus \{a\}$ by assumption. Now let $a_{*}$ be the predecessor of $a$ in $L$ (if such an element does not exist, then $a$ is minimal in $L$ and hence $L\setminus \{a\}$ is greedy and we are done). Then $a\nsim a_{*}\nsim a^{*}$. Indeed, the first incomparability follows from $a$ is minimal in $P$ and $a_*<a$ in $L$ and the second incomparability follows from $a^{*}$ is minimal in $P\setminus \{a\}$ and $a_*<a^*$ in $L\setminus \{a\}$. Now suppose for a contradiction that $L\setminus \{a\}$ is not greedy. Since the successor of $a^*$ of $a_*$ in $L\setminus \{a\}$ is minimal in $P\setminus \{a\}$ we infer from Definition \ref{def:greedy} that $U(a_*)$ has a minimal element $z$ in $P\setminus\{a\}$. It follows that $a_{*}$ is the unique lower cover of $z$ in $P\setminus \{a\}$. Since $a\nsim a_*$ and all upper covers of $a$ in $P$ are minimal in $P\setminus \{a\}$ it follows that $z$ is incomparable to $a$ in $P$ and hence $z$ must have appeared immediately after $a_{*}$ in $L$ which is not the case contradicting our assumption and proving our claim.

Denote by $\mathcal{G}(P:x<y)$ the set of greedy linear extensions of $P$ that put $x$ before $y$. We will exhibit a
bijection from $\mathcal{G}(P:x<y)$ onto $\mathcal{G}(P\setminus \{a\}:x<y)$ where $a\not \in \{x,y\}$. Let $L$ be a
greedy linear extension of $P$ that puts $x$ before $y$. Consider the map $\phi$ from $\mathcal{G}(P:x<y)$ to
$\mathcal{G}(P\setminus \{a\}:x<y)$ that maps $L$ to $L\setminus \{a\}$. Then $\phi$ is well defined. Moreover $\phi$ is onto. Indeed, if $L'$ is an element of $\mathcal{G}(P\setminus \{a\}:x<y)$, then let $a^{*}$ be the least element in $L'$ being an upper cover of $a$ (in $P$) and insert $a$ immediately before $a^{*}$ to obtain a linear extension $L$ in which $x<y$ and which is clearly
greedy because if $a_*$ is the predecessor of $a^*$ in $L'$, then $a\nsim a_*\nsim a^*$. Next we prove that $\phi$ is
one-to-one. Let $L$ and $M$ be two greedy linear extensions of $P$ that put $x$ before $y$. We view $L$ and $M$ as two order preserving
bijections from $P$ onto the total order $1<\cdots <n$ where $n$ is the number of elements of $P$. Let $i$ and $j$ be such that $L(a)=i$ and $M(a)=j$. Let $a^{*}$ be the element of $P$ such that $L(a^{*})=i+1$. Since $L$ is greedy, $a^{*}$ is an upper cover of
$a$ in $P$. Then $M(a^{*})>M(a)=j$ since $a< a^{*}$ in $P$.

Suppose that $\phi(L)=\phi(M)$, that is, $L\setminus \{a\}=M\setminus \{a\}$.
Then $L\setminus \{a\}(x)=M\setminus \{a\}(x)$ for all $x\neq a$. Note that $L\setminus\{a\}(x)=L(x)$ for all $x<a$ in $L$ and $L\setminus\{a\}(x)=L(x)-1$ for all $x>a$. Hence,
$M(a^*)-1=M\setminus\{a\}(a^*)=L\setminus\{a\}(a^*)=L(a^*)-1$. Therefore, $M(a^*)=L(a^*)$. By symmetry we show that if $b$ is the element of $P$ such that $M(b)=M(a)+1$, then $L(b)=M(b)$. Clearly we have $M(b)\leq M(a^*)$ and $L(a^*)\leq L(b)$. Hence, $M(b)\leq M(a^*)=L(a^*)\leq L(b)=M(b)$ and therefore $L(a^*)=L(b)$. Since $L$ is a bijection we infer that $b=a^*$. Finally we have that $L(a)=M(a)$ and we are done. This proves that $\phi$ is one-to-one.
\end{proof}

\begin{cor}\label{nfree3} Let $P$ be an ordered set (not necessarily $N$-free) and suppose that $P$ has a
minimal element $a$ such that $a$ is not maximal in $P$ and every
upper cover of $a$ is minimal in $P\setminus \{a\}$. Then
$|\mathcal{G}(P)|=|\mathcal{G}(P\setminus \{a\})|$.
\end{cor}

We now proceed to the proof of Theorem \ref{nfree}.

\begin{proof}(Of Theorem \ref{nfree})
The proof is by induction on the number of elements of the ordered set. Let $P$ be an $N$-free
ordered set not a chain. If $P$ has exactly two elements, then $P$ is an antichain in which case $P$ satisfies the conclusion of the
theorem. Next we suppose that $P$ has at least three elements and is neither a chain nor is an antichain. If every minimal element $a\in P$ is either maximal in $P$ or $a$ has an upper cover which is not minimal in $P\setminus \{a\}$, then it follows from Lemma \ref{nfree4} that $P$ has two minimal elements $x,y$ such that $\{x,y\}$ is autonomous in $P$. We then deduce from Corollary \ref{l2} that $P$ has a $\frac{1}{2}$-greedy balanced pair. Else, $P$ has a minimal element $a$ such that $a$ is not maximal in $P$ and every upper cover of $a$ is minimal in $P\setminus \{a\}$ and consider $P\setminus \{a\}$ which is $N$-free, not a chain and has fewer elements than $P$. By the induction hypothesis, $P\setminus \{a\}$ has a $\frac{1}{2}$-greedy balanced pair. From Lemma \ref{nfree2} we deduce that a $\frac{1}{2}$-greedy balanced pair in $P\setminus \{a\}$ remains $\frac{1}{2}$-greedy balanced in $P$. This completes the proof of the theorem.
\end{proof}

%\noindent \textbf{Acknowledgments}\\
%The author would like to thank two anonymous referees for they comments on the paper.

\end{document}